\documentclass{amsart}

\usepackage{hyperref}

\usepackage{amssymb}
\usepackage{amsmath}
\usepackage{amsthm}
\usepackage{enumerate}
\usepackage{graphicx}

\theoremstyle{plain}

\makeatletter
\newtheorem*{rep@theorem}{\rep@title}
\newcommand{\newreptheorem}[2]{%
\newenvironment{rep#1}[1]{%
 \def\rep@title{#2 \ref{##1}}%
 \begin{rep@theorem}}%
 {\end{rep@theorem}}}
\makeatother

\newtheorem{theorem}{Theorem}[section]
\newreptheorem{theorem}{Theorem}
\newtheorem{proposition}[theorem]{Proposition}

\newtheorem{corollary}[theorem]{Corollary}

\theoremstyle{definition}

\theoremstyle{remark}

\newtheorem*{acknowledgments}{Acknowledgments}

\newcommand{\UnitDisk}{\mathbb{D}}

\newcommand{\RiemannSphere}{\widehat{\mathbb{C}}}

\renewcommand {\tilde} {\widetilde}

\begin{document}

\title{Shapes, fingerprints and rational lemniscates}

\date{June 12 2014}

\author[M. Younsi]{Malik Younsi}
\address{D\'epartement de math\'ematiques et de statistique, Pavillon Alexandre--Vachon, $1045$ av. de la M\'edecine, Universit\'e Laval, Qu\'ebec (Qu\'ebec), Canada, G1V 0A6.}
\email{malik.younsi.1@ulaval.ca}


\begin{abstract}
It has been known for a long time that any smooth Jordan curve in the plane can be represented by its so-called fingerprint, an orientation preserving smooth diffeomorphism of the unit circle onto itself. In this paper, we give a new, simple proof of a theorem of Ebenfelt, Khavinson and Shapiro stating that the fingerprint of a polynomial lemniscate of degree $n$ is given by the $n$-th root of a Blaschke product of degree $n$ and that conversely, any smooth diffeomorphism induced by such a map is the fingerprint of a polynomial lemniscate of the same degree. The proof is easily generalized to the case of rational lemniscates, thus solving a problem raised by the previously mentioned authors.
\end{abstract}

\maketitle

\section{Introduction : shapes and fingerprints}

Let $\Gamma$ be a $C^\infty$ Jordan curve in the complex plane $\mathbb{C}$. Such objects are called two-dimensional \textit{shapes}. Let $\Omega_{-}$ and $\Omega_{+}$ denote the bounded and unbounded components of $\RiemannSphere \setminus \Gamma$ respectively, where $\RiemannSphere$ is the Riemann sphere. Then $\Omega_{-}$ and $\Omega_{+}$ are simply connected domains, so by the Riemann mapping theorem there exist conformal maps $\phi_{-}:\mathbb{D} \to \Omega_{-}$ and $\phi_{+}: \mathbb{D}_{+} \to \Omega_{+}$, where $\mathbb{D}=\{z \in \mathbb{C}: |z|<1\}$ is the open unit disk and $\mathbb{D}_{+}:= \RiemannSphere \setminus \overline{\mathbb{D}}$. The map $\phi_{+}$ is uniquely determined by the normalization $\phi_{+}(\infty)=\infty$ and $\phi_{+}'(\infty)>0$, the latter meaning that $\phi_{+}$ has a Laurent development of the form
$$\phi_{+}(z)=az+ \sum_{k=0}^\infty \frac{a_k}{z^k}$$
near $\infty$, for some $a>0$.

It is well-known that $\phi_{+}$ and $\phi_{-}$ extend to $C^{\infty}$ diffeomorphisms on their respective domains. Therefore, we can consider the map $k:=\phi_{+}^{-1} \circ \phi_{-} : \mathbb{T} \to \mathbb{T}$, the so-called \textit{fingerprint} of $\Gamma$, an orientation preserving $C^{\infty}$ diffeomorphism of the unit circle $\mathbb{T}$ onto itself. Note that $k$ is uniquely determined by $\Gamma$ up to postcomposition with an automorphism of $\mathbb{D}$ onto itself, i.e. a map $\phi$ of the form
$$\phi(z) = e^{i\theta} \frac{z-\alpha}{1-\overline{\alpha}z} \qquad (z \in \mathbb{D}),$$
for some real $\theta$ and some $\alpha \in \mathbb{D}$. Moreover, the fingerprint $k$ is invariant under translations and scalings of the curve $\Gamma$; in other words, if $\tilde{\Gamma}:=T(\Gamma)$ where $T(z)=az+b$ ($a>0, b \in \mathbb{C}$), then $\tilde{\Gamma}$ and $\Gamma$ have the same fingerprint.

This yields a map $\mathcal{F}$ from the set of equivalence classes of shapes, modulo linear maps $T$, into the set of equivalence classes of orientation preserving $C^{\infty}$ diffeomorphisms $k: \mathbb{T} \to \mathbb{T}$, modulo automorphisms of the unit disk $\phi$.

The study of two-dimensional shapes and their associated fingerprints was instigated by Kirillov \cite{KIR} in the 1980's and further developed by Sharon and Mumford \cite{SHM} in view of the applications to the field of computer vision and pattern recognition. More precisely, fingerprints appear to be a promising approach to the problem of classifying and recognizing objects from their observed silhouettes. In that regard, the following result is of particular interest :

\begin{theorem}
\label{thm Kir}
The map $\mathcal{F}$ is a bijection.
\end{theorem}
A proof of Theorem \ref{thm Kir} was given by Kirillov \cite{KIR} in 1987, although it follows from the so-called \textit{conformal welding theorem} of Pfluger \cite{PFL} based on the existence and uniqueness of solutions to the Beltrami equation. The interested reader may consult \cite[Chapter 2, Section 7]{LEV} and \cite{HAM} for additional information.

Note that the injectivity of $\mathcal{F}$ follows directly from the fact that $C^{\infty}$ Jordan curves are conformally removable. Indeed, suppose that $\Gamma$ and $\tilde{\Gamma}$ have the same fingerprint. Let $\phi_{+}, \phi_{-}, \tilde{\phi}_{+}, \tilde{\phi}_{-}$ be the corresponding conformal maps as above. Then
$$\phi_{+}^{-1} \circ \phi_{-} = \tilde{\phi}_{+}^{-1} \circ \tilde{\phi}_{-}$$
on $\mathbb{T}$, i.e.
$$\tilde{\phi}_{+} \circ \phi_{+}^{-1} = \tilde{\phi}_{-} \circ \phi_{-}^{-1}$$
on $\Gamma$. It follows that the map $\tilde{\phi}_{+} \circ \phi_{+}^{-1}$ can be extended to a homeomorphism of $\RiemannSphere$ onto itself, conformal outside $\Gamma$. By the aforementioned removability property, $\tilde{\phi}_{+} \circ \phi_{+}^{-1}$ is in fact a M\"{o}bius map, which is necessarily of the form $T(z)=az+b$ for some $a>0$ and $b \in \mathbb{C}$, in view of the normalizations of $\tilde{\phi}_{+}$ and $\phi_{+}$ near $\infty$.

Using Theorem \ref{thm Kir}, Sharon and Mumford described a numerical method enabling in practice to recover $\Gamma$ from its associated fingerprint $k$ and vice versa. Roughly speaking, the method consists of approximating $\Gamma$ by a polygonal curve and then using the Schwarz-Christoffel toolbox of Driscoll and Trefethen to compute the Riemann maps. See also \cite{MAR} for more numerical examples of fingerprints and shapes.

Another interesting approach is the one introduced by Ebenfelt, Khavinson and Shapiro \cite{EKS} based on approximating shapes by \textit{polynomial lemniscates}, that is sets of the form $\{ z \in \RiemannSphere : |P(z)|=1\}$, where $P$ is a polynomial. This is motivated by a classical result of Hilbert (see e.g. \cite[Chapter 4]{WAL}) saying that such lemniscates are dense in the set of $C^{\infty}$ Jordan curves, with respect to the Hausdorff metric. Moreover, the fingerprint of a polynomial lemniscate of degree $n$ is particularly simple: it is an $n$-th root of a Blaschke product of degree $n$. Conversely, the $n$-th root of any Blaschke product of degree $n$ arises as the fingerprint of a polynomial lemniscate of the same degree (see \cite[Theorem 3.1]{EKS}). Note that this gives an intuitive explanation of why Theorem \ref{thm Kir} holds, since every orientation preserving $C^{\infty}$ diffeomorphism of the unit circle onto itself can be approximated in the $C^{1}$ norm by roots of Blaschke products (see \cite[Theorem 2.3]{EKS}).

The goal of the present paper is to generalize this to the case of rational lemniscates. More precisely, it was conjectured at the end of \cite{EKS} that fingerprints of rational lemniscates of degree $n$ are precisely the $C^{\infty}$ diffeomorphisms $k:\mathbb{T} \to \mathbb{T}$ arising from the functional equation $A \circ k = B$, where $A$ and $B$ are Blaschke products of degree $n$. We give a proof of this conjecture in section \ref{secrat}, based on a new and simpler proof of the corresponding result for polynomial lemniscates (\cite[Theorem 3.1]{EKS}), which is described in detail in section \ref{secpol}.

\section{Polynomial Lemniscates}
\label{secpol}

Let $P$ be a (complex) polynomial of degree $n$. The \textit{lemniscate} of $P$, noted $\Gamma(P)$, is defined by
$$\Gamma(P):= \{ z \in \RiemannSphere : |P(z)|=1\}.$$
Following our preceding notation, we also define
$$\Omega_{-}(P):= \{z \in \RiemannSphere : |P(z)|<1\}$$
and
$$\Omega_{+}(P):= \{z \in \RiemannSphere : |P(z)|>1\}.$$
Note that by the maximum modulus principle, $\Omega_{+}(P)$ is connected. Following \cite{EKS}, we shall say that the lemniscate $\Gamma(P)$ is \textit{proper} if $\Omega_{-}(P)$ is connected. In this case, $\Omega_{-}(P)$ is simply connected, since its complement is connected.

The starting point of the study of the fingerprints of polynomial lemniscates in \cite{EKS} is the following characterization of proper lemniscates, an easy consequence of the Riemann-Hurwitz formula as found in e.g. \cite[Section 10.2]{SHE}.

\begin{proposition}
\label{proper}
Let $P$ be a polynomial of degree $n$. The following are equivalent:

\begin{enumerate}
\item The lemniscate $\Gamma(P)$ is proper.
\item All the $n-1$ critical values of $P$ (counted with multiplicites) belong to the unit disk $\mathbb{D}$.
\end{enumerate}
\end{proposition}

Consider now a polynomial $P$ of degree $n$ and assume that its corresponding lemniscate $\Gamma(P)$ is proper. Clearly, we can assume without loss of generality that the degree $n$ coefficient of $P$ is positive. Then $\Gamma(P)$ is a $C^{\infty}$ Jordan curve, and thus yields a fingerprint $k:\mathbb{T} \to \mathbb{T}$. The following theorem characterizes exactly which orientation preserving diffeomorphisms $k : \mathbb{T} \to \mathbb{T}$ are obtained in this way :

\begin{theorem}[Ebenfelt, Khavinson and Shapiro \cite{EKS}]
\label{Pfingerprints}
The fingerprint $k: \mathbb{T} \to \mathbb{T}$ of $\Gamma(P)$ is given by
$$k(z)=B(z)^{1/n},$$
where $B$ is the Blaschke product of degree $n$
$$B(z) = e^{i\theta} \prod_{k=1}^n \frac{z-a_k}{1-\overline{a_k}z}$$
for some real number $\theta$, where $a_k:=\phi_{-}^{-1}(\xi_k)$ and $\xi_1, \dots, \xi_n$ are the zeros of $P$, counted with multiplicities.

Conversely, given any Blaschke product $B$ of degree $n$, there is a polynomial $P$ of the same degree whose lemniscate $\Gamma(P)$ is proper and has $k=B^{1/n}$ as its fingerprint. Moreover, $P$ is unique up to postcomposition with a linear map of the form $T(z)=az+b$, where $a>0$ and $b \in \mathbb{C}$.
\end{theorem}

The proof of the first part of Theorem \ref{Pfingerprints} in \cite{EKS} is elementary, but we reproduce it for the reader's convenience.

Let $\phi_{-}:\mathbb{D} \to \Omega_{-}(P)$ be a Riemann map. Then $P \circ \phi_{-}$ is a degree $n$ proper holomorphic map of $\mathbb{D}$ onto itself, and hence must be a Blaschke product of degree $n$, say
$$B(z) = e^{i\theta} \prod_{k=1}^n \frac{z-a_k}{1-\overline{a_k}z}$$
for some real number $\theta$, where $a_k=\phi_{-}^{-1}(\xi_k)$.

Now, let $\phi_{+}: \mathbb{D}_{+} \to \Omega_{+}(P)$ be the conformal map normalized by $\phi_{+}(\infty)=\infty$ and $\phi_{+}'(\infty)>0$. Then $P \circ \phi_{+}$ is a degree $n$ proper holomorphic map of $\mathbb{D}_{+}$ onto itself, and hence must also be a Blaschke product, having all of its poles at $\infty$. Therefore, $P \circ \phi_{+}(z) = cz^n$ for some unimodular constant $c$. Since $\phi_{+}'(\infty)>0$ and the highest degree coefficient of $P$ is positive, we get that $c=1$. This completes the proof of the first part of Theorem \ref{Pfingerprints}.

The proof of the second part in \cite{EKS} is much more complicated and essentially relies on Koebe's continuity method based on Brouwer's invariance of domain theorem.

Here we give a simpler proof based on the following result, which seems to be of enough independent interest to be stated separately :

\begin{theorem}
\label{thmfactorP}
Let $B$ be a Blaschke product of degree $n$. Then there exist a polynomial $P$ of degree $n$ and a conformal map $\phi_{-}:\mathbb{D} \to \Omega_{-}(P)$ such that
$$B=P \circ \phi_{-}$$
on $\mathbb{D}$.
\end{theorem}

\begin{proof}
Let $A(z):=z^n$. Then $A$ and $B$ are both covering maps of degree $n$ of $\mathbb{T}$ onto itself. It follows from the basic theory of covering spaces (see e.g. \cite[Section 1.3]{HAT}) that there exist a homeomorphism $C:\mathbb{T} \to \mathbb{T}$ such that $A \circ C = B$. Clearly, $C$ extends analytically to a neighborhood of the unit circle. Consider the Riemann surface $X:=\overline{\mathbb{D}} \sqcup(\RiemannSphere \setminus \mathbb{D}) / \sim_C$ formed by welding conformally a copy of $\RiemannSphere \setminus \mathbb{D}$ to the unit disk using the analytic homeomorphism $C$ on $\mathbb{T}$. Topologically, $X$ is the connected sum of $\overline{\mathbb{D}}$ with the closed disk $\RiemannSphere \setminus \mathbb{D}$, so it is homeomorphic to a sphere. By the uniformization theorem, there exist a biholomorphism $g:X \to \RiemannSphere$ with $g(\infty)=\infty$.

Now, define $F : X \to \RiemannSphere$ by
$$
F(z) := \left\{ \begin{array}{rl} B(z)  &  \mbox{for } z\in \overline{\mathbb{D}} \\ A(z) & \mbox{for }z \in \RiemannSphere \setminus \mathbb{D} \end{array} \right..
$$
Note that the map $F$ is well-defined since $A \circ C = B$ on $\mathbb{T}$. Furthermore, $F$ is holomorphic on $X$, by Morera's theorem. The composition $F \circ g^{-1} : \RiemannSphere \to \RiemannSphere$ is meromorphic and thus equal to a rational map $R$. By construction, $F$ has exactly one pole, at $\infty$, of multiplicity $n$. It follows that $R=P$, a polynomial of degree $n$. On $\mathbb{D}$, we have $P \circ g = F = B$. Moreover,
$g^{-1}(P^{-1}(\mathbb{D})) = F^{-1}(\mathbb{D})=\mathbb{D}$, so that $g(\mathbb{D})=P^{-1}(\UnitDisk)=\Omega_{-}(P)$. Hence the result follows by letting $\phi_{-}:=\left.g\right|_\mathbb{D}$.

\end{proof}

The second part of Theorem \ref{Pfingerprints} follows directly from Theorem \ref{thmfactorP}. Indeed, let $B$ be a Blaschke product of degree $n$ and consider $P$ and $\phi_{-}:\mathbb{D} \to \Omega_{-}(P)$ as in Theorem \ref{thmfactorP}. Without loss of generality, we can assume that the degree $n$ coefficient of $P$ is positive. Then $\Omega_{-}(P)$ is connected and so $\Gamma(P)$ is a proper lemniscate. Let $\phi_{+}: \mathbb{D}_{+} \to \Omega_{+}(P)$ be the conformal map normalized by $\phi_{+}(\infty)=\infty$ and $\phi_{+}'(\infty)>0$. Then, by the same argument as in the proof of the first part of the theorem, we get that $\phi_{+}^{-1}=P^{1/n}$ and hence $k=\phi_{+}^{-1} \circ \phi_{-} = B^{1/n}$. Finally, the uniqueness part follows directly from Theorem \ref{thm Kir}.

As noted in \cite{EKS}, the preceding results yield interesting information about the critical values of Blaschke products. Indeed, if $B$ is a Blaschke product and $P$ is as in Theorem \ref{thmfactorP}, then $B$ and $P$ have the same critical values in $\mathbb{D}$. Using this and \cite[Theorem 1.2]{BCN} on the number of polynomials sharing the same critical values, one can prove :

\begin{corollary}[Ebenfelt, Khavinson and Shapiro \cite{EKS}]
For $n \geq 3$ and $w_1,\dots,w_{n-1} \in \mathbb{D}$, there are $n^{n-3}$ Blaschke products of degree $n$, counted with multiplicites and modulo automorphisms of the unit disk, whose critical values in $\mathbb{D}$ are $w_1,\dots,w_{n-1}$. For $n=2$, there is only one equivalence class.
\end{corollary}

\begin{proof}
See \cite[Corollary 3.2]{EKS}.
\end{proof}

\section{Rational Lemniscates}
\label{secrat}
Let $R$ be a rational function of degree $n$. Define $\Gamma(R), \Omega_{-}(R)$ and $\Omega_{+}(R)$ as in the case of polynomials. Assume for simplicity that $R(\infty)=\infty$, so that $\infty \in \Omega_{+}(R)$. Note that unlike the case of polynomial lemniscates, $\Omega_{+}(R)$ need not be connected. However, if we assume that $\Omega_{-}(R)$ is simply connected (and in particular connected), the rational lemniscate $\Gamma(R)$ is a shape and we can consider its fingerprint $k:\mathbb{T} \to \mathbb{T}$ defined by $k=\phi_{+}^{-1} \circ \phi_{-}$, where $\phi_{-} : \mathbb{D} \to \Omega_{-}(R)$ and $\phi_{+} : \mathbb{D}_{+} \to \Omega_{+}(R)$ are conformal maps and $\phi_+$ is normalized by $\phi_{+}(\infty)=\infty$, $\phi_{+}'(\infty) > 0$. Hence $k$ satisfies the equation
$$\phi_{+} \circ k = \phi_{-}$$ and, composing with $R$ on both sides, we obtain the equation $A \circ k = B$, where $A= R \circ \phi_{+}$ and $B=R \circ \phi_{-}$ are Blaschke products of degree $n$, as in the proof of the first part of Theorem \ref{Pfingerprints}. Furthermore, since $R(\infty)=\infty=\phi_{+}(\infty)$, we get that $A(\infty)=\infty$.

It was conjectured in \cite{EKS} that all $C^{\infty}$ diffeomorphisms $k:\mathbb{T} \to \mathbb{T}$ satisfying such functional equations are fingerprints of rational lemniscates. The main obstacle in extending the proof of Theorem \ref{Pfingerprints} in \cite{EKS} to this case is the lack of a simple analytic criterion for when the set $\Omega_{-}(R)$ is simply connected, analogous to the one in Proposition \ref{proper}. It is easy to prove using the Riemann-Hurwitz formula that a necessary condition for $\Omega_{-}(R)$ to be a Jordan domain is that $n-1$ critical values of $R$ lie in the unit disk and the remaining $n-1$ critical values lie in the complement of the closed unit disk. Unfortunately, this condition is not sufficient.

However, this difficulty can be circumvented by using the following analogue of Theorem \ref{thmfactorP} :

\begin{theorem}
\label{thmfactorR}
Let $A$ and $B$ be Blaschke products of degree $n$ with $A(\infty)=\infty$. Then there exist a rational map $R$ of degree $n$ with $R(\infty)=\infty$ and conformal maps $\phi_{-}:\mathbb{D} \to \Omega_{-}(R)$ and $\phi_{+} : \mathbb{D}_{+} \to \Omega_{+}(R)$ normalized by $\phi_{+}(\infty)=\infty$, $\phi_{+}'(\infty)>0$ such that
$$A=R \circ \phi_{+}$$
on $\mathbb{D}_{+}$
and
$$B=R \circ \phi_{-}$$
on $\mathbb{D}$.
\end{theorem}

\begin{proof}
The argument is quite similar to the one used in Theorem \ref{thmfactorP}. Since $A$ and $B$ are both covering maps of degree $n$ of $\mathbb{T}$ onto itself, there exist an analytic homeomorphism $k:\mathbb{T} \to \mathbb{T}$ such that $A \circ k = B$. Consider the Riemann surface $X:=\overline{\mathbb{D}} \sqcup(\RiemannSphere \setminus \mathbb{D}) / \sim_k$. Again by the uniformization theorem, there exist a biholomorphism $g:X \to \RiemannSphere$ with $g(\infty)=\infty$ and $g'(\infty)>0$.

Now, define $F : X \to \RiemannSphere$ by
$$
F(z) := \left\{ \begin{array}{rl} B(z)  &  \mbox{for } z\in \overline{\mathbb{D}} \\ A(z) & \mbox{for }z \in \RiemannSphere \setminus \mathbb{D} \end{array} \right..
$$
Then $F$ is well-defined and holomorphic on $X$, by Morera's theorem. The composition $F \circ g^{-1} : \RiemannSphere \to \RiemannSphere$ is meromorphic and thus equal to a rational map $R$.

On $\mathbb{D}$, we have $R \circ g = F = B$ and on $\RiemannSphere \setminus \overline{\mathbb{D}}=\mathbb{D}_{+}$, we have $R \circ g = F = A$. The result follows by letting $\phi_{-}:=\left.g\right|_\mathbb{D}$ and $\phi_{+}:=\left.g\right|_{\mathbb{D}_{+}}$.
\end{proof}

As a consequence, we obtain :

\begin{theorem}
\label{Rfingerprints}
Let $R$ be a rational map of degree $n$ with $R(\infty)=\infty$ whose lemniscate $\Gamma(R)$ is proper. Then the fingerprint $k: \mathbb{T} \to \mathbb{T}$ of $\Gamma(R)$ is given by a solution $k$ to the functional equation
$$A \circ k = B,$$
where $A,B$ are Blaschke products of degree $n$ and $A(\infty)=\infty$.

Conversely, given any solution $k$ to a functional equation of the form
$$A \circ k = B$$
where $A,B$ are Blaschke products of degree $n$ and $A(\infty)=\infty$, there exist a rational map $R$ of degree $n$ with $R(\infty)=\infty$ whose lemniscate $\Gamma(R)$ is proper and has $k$ as its fingerprint. Moreover, $R$ is unique up to postcomposition with a linear map of the form $T(z)=az+b$, where $a>0$ and $b \in \mathbb{C}$.
\end{theorem}

\acknowledgments{The author thanks Maxime Fortier Bourque for helpful discussions as well as Eric Schippers for relevant comments on the historical background of Theorem \ref{thm Kir}.}

\bibliographystyle{amsplain}

\end{document}